\documentclass{article}

\usepackage[nobrazil]{lymber}

\usepackage{subfig, float, caption}
\usepackage{mathpazo}
\usepackage{enumerate}

\newcommand\pair[1]{\langle{#1}\rangle}
\newcommand{\sph}{\mathbb{S}}
\newcommand{\hyp}{\mathbb{H}}
\newcommand{\LM}{\mathbb{L}}

\newcommand{\keywords}[1]{\emph{Keywords: #1}}
\newcommand{\subjclass}[1]{\emph{MSC 2010: #1}}

\theoremstyle{theorem}
\newtheorem{theorem}{Theorem}[section]
\newtheorem{proposition}[theorem]{Proposition}
\newtheorem{lem}[theorem]{Lemma}

\theoremstyle{remark}
\newtheorem{remark}[theorem]{Remark}

\numberwithin{equation}{section}

\begin{document}

\title{A local characterization for constant curvature metrics in
  2-dimensional Lorentz manifolds}
\author{Ivo Terek Couto\thanks{The first author was supported by FAPESP,
    grant $2014/09781-8$} \and Alexandre Lymberopoulos}

\date{\today}

\maketitle

\begin{abstract}
  In this paper we define \emph{Fermi-type coordinates} in a
  $2-$dimensional Lorentz manifold, and use this coordinate system to
  provide a local characterization of constant Gaussian curvature
  metrics for such manifolds, following a classical result from
  Riemann. We then exhibit particular isometric immersions of such
  metrics in the pseudo-Riemannian ambients $\LM^3\equiv \R^3_1$ and
  $\R^3_2$.

  \bigskip

  \noindent\subjclass{Primary: 53B30}\\
  \noindent\keywords{Lorentz metrics, Constant Gaussian curvature, pseudo-Riemannian surfaces}
\end{abstract}

\section{Introduction}

Constant sectional curvature metrics on pseudo-Riemannian manifolds
appear naturally in physics as some of the solutions of Einstein's
equation in vacuum \[R_{\mu \nu} -\frac{1}{2}R g_{\mu \nu}
+\Lambda g_{\mu \nu}=0.\]

If the \emph{cosmological constant} $\Lambda$ is positive a particular
solution is the \emph{de Sitter} metric; if $\Lambda=0$, it is the
Lorentz-Minkowski metric; if $\Lambda$ is negative, it is the
\emph{anti-de Sitter} metric.

In $3-$dimensional Lorentz-Minkowski space, surfaces with constant
Gaussian curvature have been studied under a diverse set of hypotheses
with different techniques as in~\cite{RL},~\cite{AEG} and~\cite{GHI}.

The local classification of pseudo-Riemannian metrics of index $s$
and constant sectional curvature on $n-$dimensional manifolds is given
by the following classical result, known as Riemann's theorem
(see~\cite[pp. 69]{Wol}).

\begin{theorem}[Riemann]\label{teo:riemann}
  Let $M$ be a pseudo-Riemannian manifold of dimension $n\geq 2$ and let
  $K$ be a real number. Then the following conditions are equivalent:
  \begin{enumerate}[(i)]
    \item $M$ is of constant curvature $K$.
    \item If $x\in M$, then there are local coordinates $u^i$ on a
    neighborhood of $x$ in which the metric is given by\[{\rm
      d}s^2=\frac{\epsilon_1{\rm d}u^1\otimes {\rm
        d}u^1+\ldots+\epsilon_n{\rm d}u^n\otimes {\rm
        d}u^n}{\left(1+\frac{K}{4}\sum \epsilon_i(u^i)^2\right)^2},\quad
    \epsilon_i=\pm 1.\]
    \item If $x\in M$, then $x$ has a neighborhood which is isometric to
    an open set on some $\sph^n_s$ if $K>0$, $\R^n_s$ if $K=0$, $\hyp^n_s$
    if $K<0$.
  \end{enumerate}
\end{theorem}

In the previous statement we have
\begin{align*}
  \R^n_s&=\big(\R^n,\pair{\cdot,\cdot}_s\big)\\
  \sph^n_s&=\big\{x\in\R^{n+1}_{s}: \pair{x,x}_s=1\big\}\\
  \hyp^n_s&=\big\{x\in\R^{n+1}_{s+1}: \pair{x,x}_{s+1}=-1\big\},
\end{align*} where
\[\pair{x,y}_s=\sum\limits_{i=1}^{n-s}x^iy^i-\sum\limits_{i=n-s+1}^nx^iy^i,\]
for any $x=(x^1,\ldots,x^n)$ and $y=(y^1,\ldots,y^n)$ in $\R^n$.

When $s=0$ we have $\R^n_0=\R^n$, the usual Euclidean space and
$\sph^n_0=\sph^n$, the standard Euclidean unit sphere. Furthermore, if
$s=1$ we have $\R^n_1=\LM^n$, the Lorentz-Minkowski space; $\sph^n_1$,
the de Sitter space (eventually denoted by $dS_n$); $\hyp^n_0=\hyp^n$,
the usual hyperbolic space with $x_{n+1}>0$ and $\hyp^n_1$ the anti-de
Sitter space (eventually denoted by $AdS_n$).

In particular, when $n=2$ and $s=0$ we have the following result.

\begin{theorem}
  A surface with $K=0$ is locally isometric to the plane, with $K=1$ is
  locally isometric to the unit sphere and with $K=-1$ is locally
  isometric to the hyperbolic plane.
\end{theorem}

In this work we show a simpler proof of the above theorem when $n=2$ and
$s=1$, using Fermi-type coordinates. More precisely we have:

\begin{theorem}
  Let $M$ be a $2-$dimensional pseudo-Riemannian manifold of constant
  curvature\linebreak $K\in\{-1,0,1\}$. Then one and only one of the following
  holds:
  \begin{enumerate}[(i)]
    \item if $K=0$ then $M$ is locally isometric to Lorentz-Minkowski
    plane $\LM^2$, with metric expressed in Fermi-type coordinates as
    ${\rm d}s^2 = -{\rm d}u^2 + {\rm d}v^2$ or ${\rm d}s^2 = {\rm d}\tau^2
    - {\rm d}\vartheta^2$;
    \item if $K=1$ then $M$ is locally isometric to the de Sitter space
    $\sph^2_1$, with metric expressed in Fermi-type coordinates as ${\rm
      d}s^2 = -{\rm d}u^2 + \cosh^2u\,{\rm d}v^2$ or ${\rm d}s^2 = {\rm
      d}\tau^2-\cos^2\tau\,{\rm d}\vartheta^2$;
    \item if $K=-1$ then $M$ is locally isometric to the anti-de Sitter
    space $\hyp^2_1$. In this case the metric is expressed in Fermi-type
    coordinates as ${\rm d}s^2 = -{\rm d}u^2 + \cos^2u\,{\rm d}v^2$ or
    ${\rm d}s^2 = {\rm d}\tau^2-\cosh^2\tau\,{\rm d}\vartheta^2$;
  \end{enumerate}
\end{theorem}

\section{Notation and Preliminaries}

Firstly, we establish some notation and recall some standard definitions
and results from pseudo-Riemannian geometry.

A pseudo-Riemannian manifold $(M^n,\pair{\cdot,\cdot})$ where
$\pair{\cdot,\cdot}$ has index $1$ will be called a \emph{Lorentz
  manifold}. In a Lorentz manifold, a tangent vector $X_p \in T_pM$ is
\begin{itemize}
  \item \emph{spacelike} if $\pair{X_p,X_p} > 0$ or if $X_p = 0$;
  \item \emph{timelike} if $\pair{X_p,X_p} < 0$;
  \item \emph{lightlike} if $\pair{X_p,X_p} = 0$ but $X_p \neq 0$.
\end{itemize}

Every tangent vector is of exactly one of the above \emph{causal
  types}. The notion of causal type extends to curves in $M$ in a
natural way: if $\gamma\colon I \to M$ is a curve, then $\gamma$ is said
to be of one of the three causal types above if all its tangent vectors
$\gamma'(t)$ share that causal type.

Recall that a \emph{geodesic} in $M$ is a curve $\gamma\colon I \to M$
such that $\nabla_{\gamma'(t)}\gamma' = 0$ for all $t \in I$, where
$\nabla$ is the \emph{Levi-Civita connection} of the metric, or
equivalently (see~\cite[p. 67]{ON}), if in all local coordinates
$(u^1,\cdots,u^n)$ we
have \begin{equation}\label{eq:geo_local}\ddot{u}^k +
  \sum_{i,j=1}^n\Gamma_{ij}^k \dot{u}^i\dot{u}^j = 0,\mbox{ for all $k
    \in \{1,\ldots,n\}$,}
\end{equation}
where the dots denote the derivatives of the coordinates of $\gamma$
with respect to the curve parameter and the $\Gamma_{ij}^k$ are the
\emph{Christoffel symbols} of the metric in the coordinates chosen. We
also have that geodesics do not change causal type, since \[ \frac{{\rm
    d}}{{\rm d}t}\pair{\gamma'(t),\gamma'(t)} =
2\pair{\gamma'(t),\nabla_{\gamma'(t)}\gamma'} = 0.  \]It is also easy to
see that any non-lightlike curve has a unit-speed reparametrization.

If $M$ is $2-$dimensional, and $(u,v)$ are local coordinates in $M$, we
express the metric $\pair{\cdot,\cdot}$ in these coordinates as: \[ {\rm
  d}s^2 = E\,{\rm d}u^2 + 2F\,{\rm d}u\,{\rm d}v + G\,{\rm
  d}v^2,\]where \[ E = \pair{\partial_u,\partial_u}, \quad F =
\pair{\partial_u,\partial_v}, \quad G =
\pair{\partial_v,\partial_v} \]are the \emph{the coefficients of the
  metric (first fundamental form) in the coordinates $(u,v)$}.

On the following we establish some results about $2-$dimensional Lorentz
manifolds (henceforth called \emph{Lorentz surfaces}) starting with an
analogous result to the existence of orthogonal coordinates for
Riemannian surfaces.

\begin{lem}\label{lem:local_coords}
  Let $X$ and $Y$ be linearly independent vector fields in a
  neighborhood of a point $p$ in a $2-$di\-men\-sional smooth
  manifold. There exists a coordinate system $(u,v)$ at $p$ such that
  $\partial_u=\lambda X$ and $\partial_v=\mu Y$.
\end{lem}

\begin{proof}
  We are looking for smooth positive functions $\lambda$ and $\mu$ such
  that the Lie bracket $[\lambda X, \mu Y]$ vanishes, since by
  Frobenius' theorem we will be able to find local coordinates $(u,v)$
  in $M$ satisfying $\partial_u=\lambda X$ and $\partial_v=\mu
  Y$. Writing $[X,Y] = fX+gY$ for suitable smooth functions $f$ and $g$
  and using the standard properties of Lie brackets leads to
  \[0=[\lambda X,\mu Y]=\big(\lambda\mu f -\mu
  Y(\lambda)\big)X+\big(\lambda\mu g+\lambda X(\mu)\big)Y,\] that is
  \begin{equation}
    \label{eq:pdesys}
    Y(\lambda)=\lambda f\quad\mbox{and}\quad -X(\mu)=\mu g.
  \end{equation}

  If $(x,y)$ is any coordinate system at $p$ then we may write
  $X=a\partial_x+b\partial_y$ and $Y=c\partial_x+d\partial_y$, where
  $a,b,c$ and $d$ are smooth functions such that $ad-bc\neq 0$ and
  $(a^2+b^2)(c^2+d^2)\neq 0$. Hence we may
  write \[X(\mu)=a\mu_x+b\mu_y\quad\mbox{and}\quad
  Y(\lambda)=c\lambda_x+d\lambda_y\] so that equations (\ref{eq:pdesys})
  become \[
  \begin{cases}
    c\dfrac{\lambda_x}{\lambda}+d\dfrac{\lambda_y}{\lambda}=f,\\[0.5em]
    a\dfrac{\mu_x}{\mu}+b\dfrac{\mu_y}{\mu}=g.
  \end{cases}\]

  Let $h_1(x,y)=\ln\big(\lambda(x,y)\big)$ and
  $h_2(x,y)=\ln\big(\mu(x,y)\big)$. With this, the equations above are
  written as 
  \[
  \begin{cases}
    c(x,y)(h_1)_x(x,y)+d(x,y)(h_1)_y(x,y)=f(x,y),\\[0.5em]
    a(x,y)(h_2)_x(x,y)+b(x,y)(h_2)_y(x,y)=g(x,y),
  \end{cases}\] which are two linear first order PDEs, that have
  solutions due to the condition $ad-bc\neq 0$. Hence we have found
  $\lambda=e^{h_1}$ and $\mu=e^{h_2}$ as desired.
\end{proof}

In the case of a Lorentz surface $M$, lemma \ref{lem:local_coords}
ensures that we can, at least locally, write the metric on $M$ as
\begin{equation}
  \label{eq:met_loc}
  {\rm d}s^2=E\,{\rm d}u^2+G\,{\rm d}v^2,
\end{equation}
where $E$ and $G$ are smooth functions such that $EG<0$. To see this,
apply the lemma for two non-lightlike and non-zero orthogonal vector
fields of opposite causal characters.

\begin{lem}
  In this setting the Gaussian curvature of the metric is
  \begin{equation}
    \label{eq:K_orthog}
    K=
    \frac{-1}{\sqrt{|EG|}}\left(\epsilon_1\left(\frac{(\sqrt{|G|})_u}{\sqrt{|E|}}\right)_u+
      \epsilon_2\left(\frac{(\sqrt{|E|})_v}{\sqrt{|G|}}\right)_v\right),
  \end{equation}
  where $\epsilon_1$ (resp. $\epsilon_2$) is $1$ or $-1$ if $\partial_u$
  (resp. $\partial_v$) is spacelike or timelike.
\end{lem}

For a proof of the lemma above see~\cite[p. 54]{ON'}.

\section{Fermi-type coordinates and its properties}

From equation (\ref{eq:geo_local}) one sees that given a point $p\in M$
and any unit vector $w\in T_pM$ there exists a unique geodesic
$\gamma:I\to M$ passing through $p$ with velocity $w$.

We can use this fact to exhibit a local parametrisation of $M$ known, in
Euclidean case, as \emph{Fermi parametrisation} (see~\cite[p. 275]{Opr}
and~\cite[p. 192]{Bar}). For a Lorentz surface we have distinct
Fermi-type parametrisations, according to the causal type of the vector
$w=\gamma'(0)\in T_pM$. In order to obtain a regular parametrisation we
avoid lightlike geodesics. The construction is the same in both
remaining cases and it is done as follows:

Fix an unit speed geodesic $\gamma\colon I \to M$. For each
$\gamma(v)\in M$ consider the unit speed geodesic $\gamma_v:J_v\to M$
intersecting $\gamma$ orthogonally in $\gamma_v(0)=\gamma(v)$. Set the
map ${\bf x}$ as ${\bf x}(u,v) = \gamma_v(u)$, for each $v\in I$ and
$u\in J_v$.

Recall that geodesics have constant causal type, that is, if $\gamma$ is
spacelike (resp. timelike) at some point, $\gamma$ will be spacelike
(resp. timelike) everywhere. Therefore all of the $\gamma_v$ will be
timelike (resp. spacelike), since $\big\{\gamma'(v),\gamma'_v(0)\big\}$
is an orthonormal basis of $T_{\gamma(v)}M$, for all $v\in I$.

\begin{proposition}\label{prop:regular}
  ${\bf x}$ is regular in some open neighborhood of $\gamma(I)\subset
  M$.
\end{proposition}

\begin{proof}
  We show that ${\bf x}_u$ e ${\bf x}_v$ are linearly independent in
  some neighborhood of $(0,v)$, for all $v\in I$. It suffices to note
  that \[ F(0,v) = \langle {\bf x}_u(0,v),{\bf x}_v(0,v)\rangle =
  \langle \gamma_v'(0), \gamma'(v)\rangle = 0. \] Then ${\bf
    x}_u(0,v)$ e ${\bf x}_v(0,v)$ are orthogonal along $\gamma(I)$,
  hence linearly independent (since they are not lightlike) in some open
  neighborhood of it in $M$.
\end{proof}

\begin{proposition}\label{prop:metric}
  In this coordinate system the metric of $M$ is given by
  \begin{equation}
    \label{eq:metric}
    {\rm d}s^2 = -\epsilon_{\gamma}\,{\rm d}u^2 + G\,{\rm d}v^2
  \end{equation}
  where $\epsilon_\gamma$ is $1$ if $\gamma$ is spacelike or $-1$ if
  timelike.
\end{proposition}

\begin{proof}
  All $\gamma_v$ are unit speed curves with the same causal type
  $\epsilon_{\gamma_v}$. We have \[ E(u,v) = \langle {\bf x}_u(u,v),{\bf
    x}_u(u,v)\rangle = \langle \gamma'_v(u),\gamma'_v(u)\rangle =
  \epsilon_{\gamma_v}.\] Now $\epsilon_\gamma \epsilon_{\gamma_v} = -1$
  for all $v$, since $M$ is $2-$dimensional, the metric is
  non-degenerate and has index $1$. Hence $E(u,v) =
  -\epsilon_\gamma$.

  Also, by construction, $F(0,v) = 0$ for all $v$. It remains to show
  that $F$ does not depend on $u$. Fixing $v_0\in I$, note that
  $\gamma_{v_0}(u) = {\bf x}(u,v_0)$ has coordinates $u^1 = u$ e $u^2 =
  v_0$, so making $k = 2$ in (\ref{eq:geo_local}) yields $\Gamma_{11}^2
  = 0$.
  
  Writing $(g_{ij})=\begin{pmatrix} E&F\\F&G
  \end{pmatrix}$ and its inverse as $(g^{ij})$, the expressions of
  Christoffel in terms of these objects are (see~\cite[p. 62]{ON}):
  \begin{equation}
    \label{eq:chris_met}
    \Gamma_{ij}^k=\sum_{r=1}^2\frac{1}{2}g^{kr}\left(\frac{\partial
        g_{ir}}{\partial u^j}+\frac{\partial g_{jr}}{\partial
        u^i}-\frac{\partial g_{ij}}{\partial
        u^r}\right).
  \end{equation}
  Setting $k=2$ and $i=j=1$ in equation (\ref{eq:chris_met}) leads to
  $g^{22}F_u(u,v_0)=0$. It is clear \linebreak that
  $g^{22}=\dfrac{-\epsilon_\gamma}{\det (g_{ij})} \neq 0$, hence
  $F_u(u,v_0)=0$, that is, $F(u,v)=0$.
\end{proof}

\begin{remark}
  Since $G(0,v)=\epsilon_\gamma\neq 0$ and $G$ is continuous, we can
  assume that the neighborhood of $\gamma(I) \subseteq M$ found on
  proposition \ref{prop:regular} (reducing it if necessary) is such that
  $G$ and $\epsilon_\gamma$ have the same sign.
\end{remark}

\begin{lem}
  In the same setting of the previous proposition, the Gaussian
  curvature of the \linebreak metric ${\rm d}s^2 = -\epsilon_\gamma\,{\rm d}u^2 +
  G\,{\rm d}v^2$ is
 \begin{equation}
   \label{eq:K}
   K=\epsilon_\gamma \frac{(\sqrt{|G|})_{uu}}{\sqrt{|G|}}.
 \end{equation}
\end{lem}

\begin{proof}
  Make $E=-\epsilon_\gamma$ in (\ref{eq:K_orthog}).
\end{proof}

\begin{lem}\label{lem:bndry_cndtns}
  $G_u(0,v) = 0$ for each $v$.
\end{lem}

\begin{proof}
  Using the same notation in the proof of proposition \ref{prop:metric},
  observing that \linebreak $\gamma(v)=\gamma_v(0)={\bf x}(0,v)$ has coordinates $u^1
  = 0$, $u^2 = u$, making $k=1$ in (\ref{eq:geo_local}) we obtain
  $\Gamma_{22}^1 = 0$. Then, setting $k=1$ and $i=j=2$ in
  (\ref{eq:chris_met}) gives $G_u(0,v)=0$, since $g^{11}=\dfrac{G}{\det
    (g_{ij})}$.
\end{proof}

\begin{remark}
  We observe that there exists two Fermi-type coordinates:
  \begin{itemize}
    \item \emph{spacelike Fermi-type} coordinates, when the fixed
    geodesic $\gamma$ is spacelike, that is, $\epsilon_\gamma=1$;
    \item \emph{timelike Fermi-type} coordinates, when the fixed
    geodesic $\gamma$ is timelike, that is, $\epsilon_\gamma=-1$.
  \end{itemize}
  To avoid confusion we will denote the parameters in the timelike
  Fermi-type coordinates by $(\tau,\vartheta)$.
\end{remark}

\section{Proof of the Main Result}

At this point we are able to prove the main result of this work,
providing a local classification of metrics with constant curvature
dependng on the chosen Fermi-type coordinates. We work with the possible
values for $K$.

\begin{enumerate}[(a)]
  \item $K = 0$: in this case equation (\ref{eq:K}) becomes
  \[\frac{(\sqrt{G})_{uu}}{\sqrt{G}} = 0\quad\mbox{or}\quad
  -\frac{(\sqrt{-G})_{\tau\tau}}{\sqrt{-G}} = 0,\] according to
  causal type of $\gamma$, whose solutions are
  $\sqrt{G(u,v)}=A(v)u+B(v)$ and \linebreak
  $\sqrt{-G(\tau,\vartheta)}=A(\vartheta)\tau+B(\vartheta)$, respectively. Since
  $G=\epsilon_\gamma$ and $G_u=G_\tau=0$ along $\gamma$ (by lemma
  \ref{lem:bndry_cndtns}), we have $A\equiv 0$ and $B\equiv 1$.
  
  We conclude that \[{\rm d}s^2 = -{\rm d}u^2 + {\rm d}v^2 = {\rm
    d}\tau^2 - {\rm d}\vartheta^2,\] and we see that $M$ is locally
  isometric to the Lorentz-Minkowski plane $\LM^2$.

  \item $K=1$: 
  \begin{itemize}
    \item if $\gamma$ is spacelike, equation (\ref{eq:K})
    becomes \[\frac{(\sqrt{G})_{uu}}{\sqrt{G}} = 1,\] whose solutions
    are of the form $\sqrt{G(u,v)} = A(v)e^u + B(v)e^{-u}$. The initial
    conditions $G(0,v)=1$ and $G_u(0,v)=0$ lead to
    $A(v)=B(v)=\frac{1}{2}$ and $G(u,v)=\cosh^2u$, that is \[{\rm d}s^2
    = -{\rm d}u^2 + \cosh^2u\,{\rm d}v^2.\]
    \item if $\gamma$ is timelike, equation (\ref{eq:K})
    becomes \[\frac{(\sqrt{-G})_{\tau\tau}}{\sqrt{-G}}=1.\]
    Solving this differential equation, with initial conditions as
    above, we obtain\linebreak $G(\tau,\vartheta)=-\cos^2\tau$ and \[{\rm d}s^2
    = {\rm d}\tau^2 - \cos^2\tau\,{\rm d}\vartheta^2.\]
  \end{itemize}

  \item $K=-1$: in this case equation (\ref{eq:K}) is very similar to
  the previous case and the solutions are $G(u,v)=\cos^2u$ e
  $G(\tau,\vartheta)=-\cosh^2\tau$. Hence \[ {\rm d}s^2 = -{\rm d}u^2
  + \cos^2u\,{\rm d}v^2\quad\mbox{or}\quad {\rm d}s^2={\rm d}\tau^2
  - \cosh^2\tau\,{\rm d}\vartheta^2.\]
\end{enumerate}

\section{Realization of those metrics}

In the previous section we provided local expressions for metrics of
constant curvature. Now we exhibit immersions of these metrics in $\LM^3$
and $\R^3_2$. Straightforward computations show that the following
immersions have the desired metrics.

\begin{enumerate}[(a)]
  \item For $K=0$ the immersions are trivially given by inclusions as
  coordinate planes.
  \item For $K=1$ we have
  \begin{itemize}
    \item ${\rm d}s^2 = -{\rm d}u^2+\cosh^2u\,{\rm d}v^2$.
    \begin{enumerate}[(i)]
      \item ${\bf x}\colon \R^2 \to \sph^2_1 \subseteq \LM^3$, \[{\bf
        x}(u,v) = (\cosh u \cos v, \cosh u \sin v, \sinh u).\]
      \item ${\bf x}\colon \cosh^{-1}\big(]1,\sqrt{2}[\big)\times \R \to \R^3_2$,
      \[{\bf x}(u,v) = \left(\cosh u \cosh v, \cosh u \sinh v,
        \int_0^u\sqrt{2-\cosh^2t}\,{\rm d}t\right).\]
    \end{enumerate}
    
    \item ${\rm d}s^2 = {\rm d}\tau^2 - \cos^2\tau\,{\rm d}\vartheta^2$. 
    \begin{enumerate}[(i)]
      \item ${\bf x}\colon\R^2 \to \sph^2_1\subseteq\LM^3$, \[{\bf
        x}(\tau,\vartheta) = (\sin\tau, \cos\tau \cosh \vartheta,
      \cos\tau \sinh \vartheta).\]
      \item ${\bf x}\colon \Big]-\frac{\pi}{2},\frac{\pi}{2}\Big[\times\R\to
      \R^3_2$, \[{\bf x}(\tau,\vartheta) = \left(\int_0^\tau
        \sqrt{1+\sin^2t}\,{\rm d}t, \cos\tau\cos\vartheta,
        \cos\tau\sin\vartheta\right).\]
    \end{enumerate}
    
    \begin{remark}\label{obs:maxper}
      Since translations are isometries in $\R^3_2$ and we have the
      periodicity condition ${\bf x}(\tau,\vartheta)={\bf
        x}(\tau+\pi,\vartheta)$, we can restrict the domain of the latter
      parametrization to the given above, which is maximal since the
      metric is singular at its boundary.
    \end{remark}
  \end{itemize}

  If $K=1$ the surface can be seen, for example, as a piece of one of
  the following surfaces (up to isometries of the ambient):

  \begin{figure}[H]
    \centering
    \subfloat[$K=1$ in $\LM^3$]{
      \includegraphics[width=.3\textwidth]{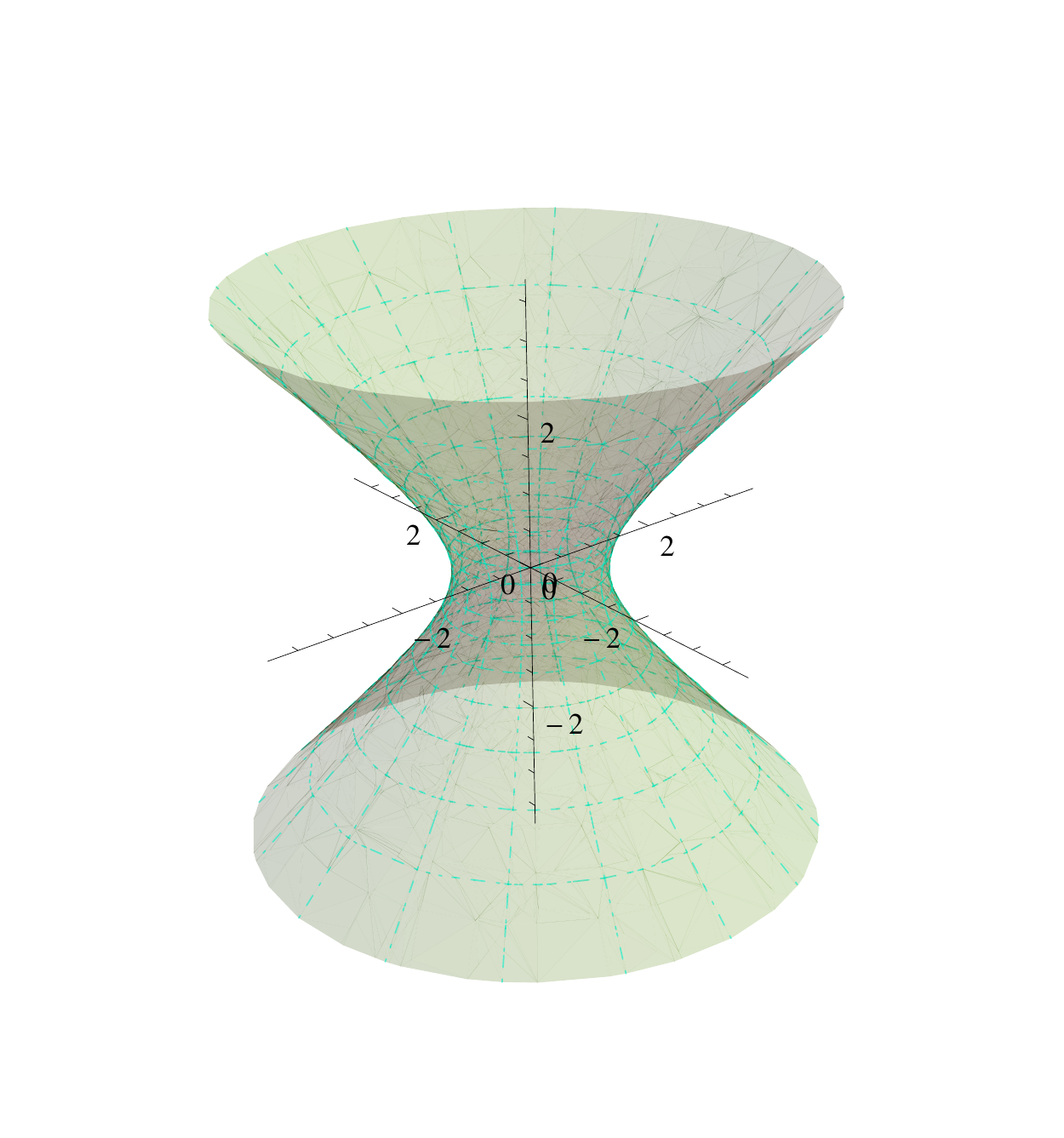}\label{fig:K1L3}}
    \ \subfloat[$K=1$ in $\R^3_2$]{
      \includegraphics[width=.3\textwidth]{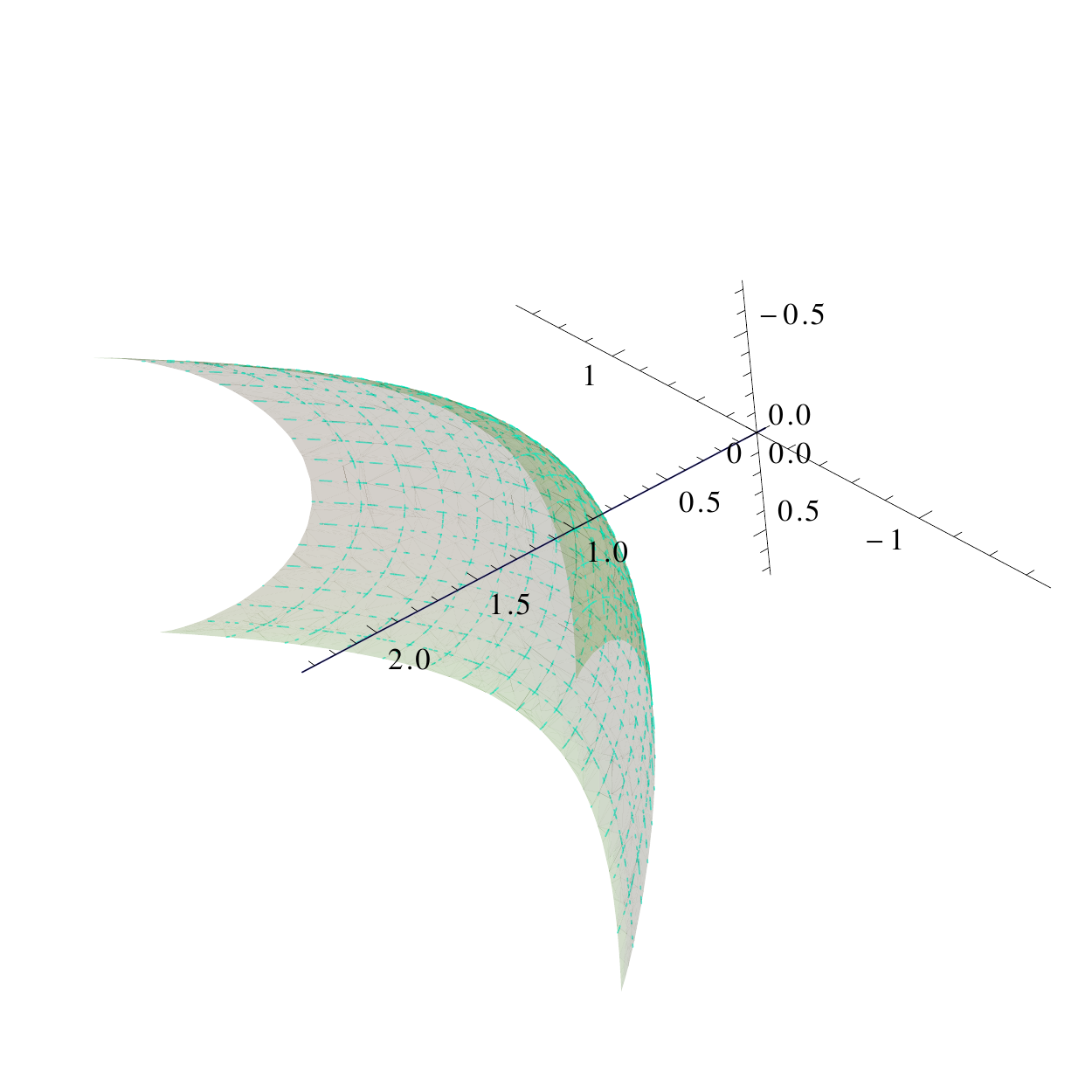} 
      \includegraphics[width=.3\textwidth]{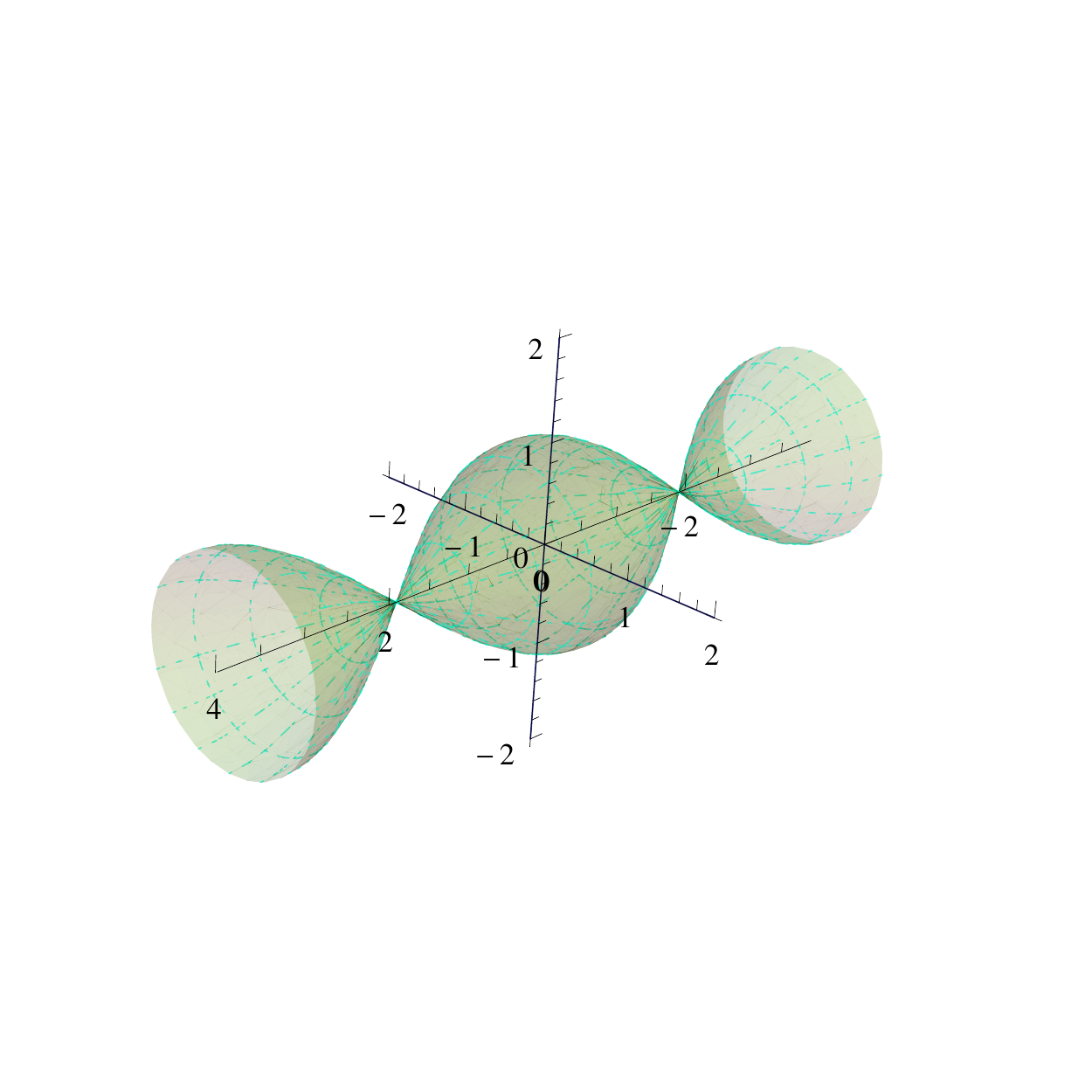}\label{fig:K1R32}}\ 
    \caption{Constant Gaussian Curvature $1$}
    \label{fig:K=1}
  \end{figure}

  \item $K=-1$:
  
  \begin{itemize}
    \item ${\rm d}s^2 = -{\rm d}u^2 + \cos^2u\,{\rm d}v^2$.
    \begin{enumerate}[(i)]
      \item ${\bf x}\colon\Big]-\frac{\pi}{2},\frac{\pi}{2}\Big[\times\R \to
      \LM^3$, \[{\bf x}(u,v) = \left(\cos u\cos v, \cos u \sin v,
        \int_0^u\sqrt{1+\sin^2t}\,{\rm d}t\right).\]
      \item ${\bf x}\colon \R^2 \to\hyp^2_1\subseteq\R^3_2$, \[{\bf x}(u,v) =
      (\cos u \sinh v, \cos u\cosh v, \sin u).\]
    \end{enumerate}
    The same situation in remark \ref{obs:maxper} holds in this case.

    \item ${\rm d}s^2 = {\rm d}\tau^2 - \cosh^2\tau\,{\rm d}\vartheta^2$.
    \begin{enumerate}[(i)]
      \item ${\bf x}\colon \cosh^{-1}\big(]1,\sqrt{2}[\big)\times \R\to\LM^3$,
      \[{\bf x}(\tau,\vartheta) = \left(\int_0^\tau
        \sqrt{2-\cosh^2t}\,{\rm
          d}t,\cosh\tau\cosh\vartheta,\cosh\tau\sinh\vartheta\right).\]
      \item ${\bf x}\colon \R^2 \to \hyp^2_1 \subseteq \R^3_2$, \[{\bf
        x}(\tau,\vartheta)=(\sinh\tau,\cosh\tau\cos\vartheta,\cosh\tau\sin\vartheta).\]
    \end{enumerate}
  \end{itemize}
  
  If $K=-1$ the surface can be seen, for example, as a piece of one of
  the following surfaces (again up to isometries of the ambient):

  \begin{figure}[H]
    \centering
    \subfloat[$K=-1$ in $\R^3_2$]{
      \includegraphics[width=.3\textwidth]{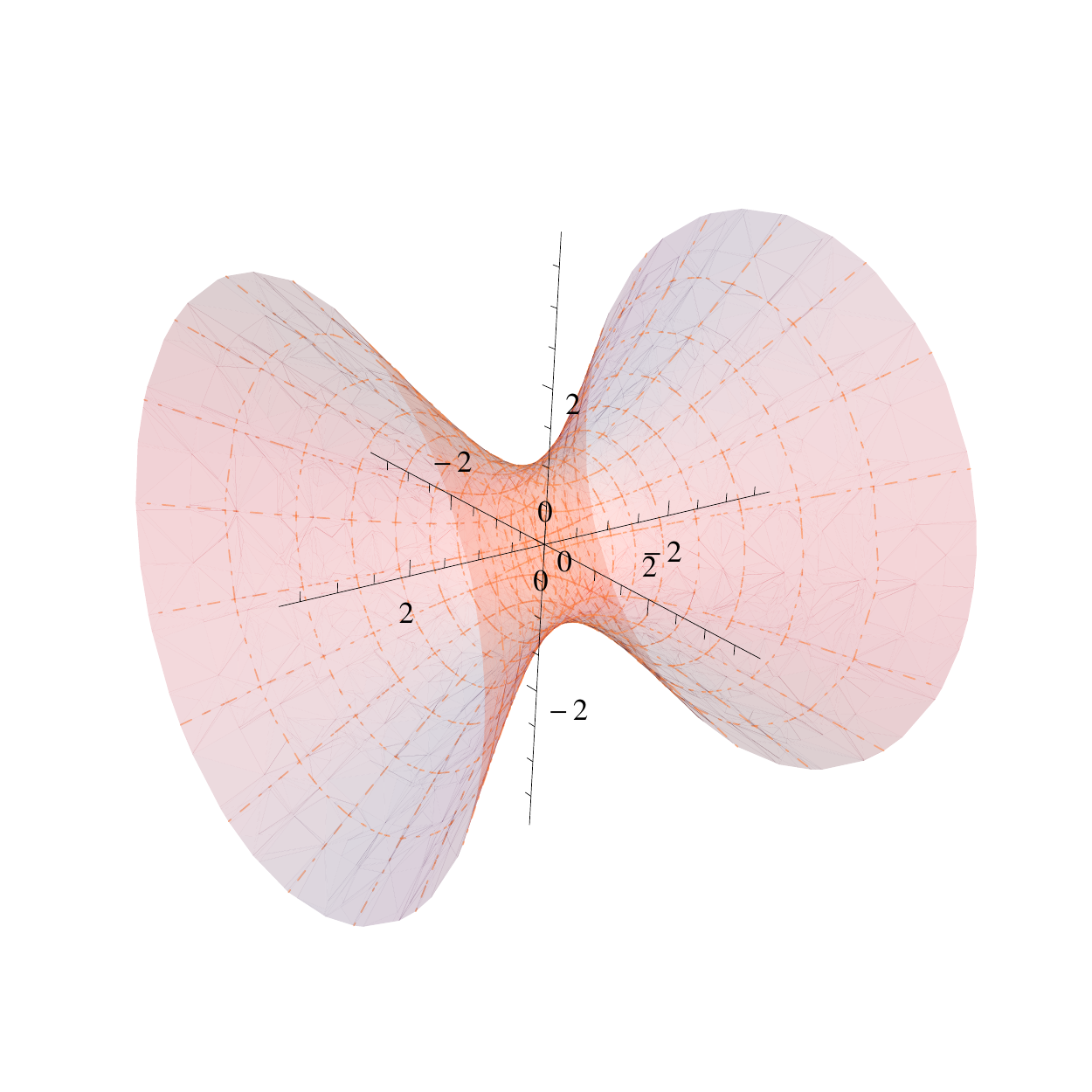}\label{fig:K-1R32}} \ 
    \subfloat[$K=-1$ in $\LM^3$]{
      \includegraphics[width=.3\textwidth]{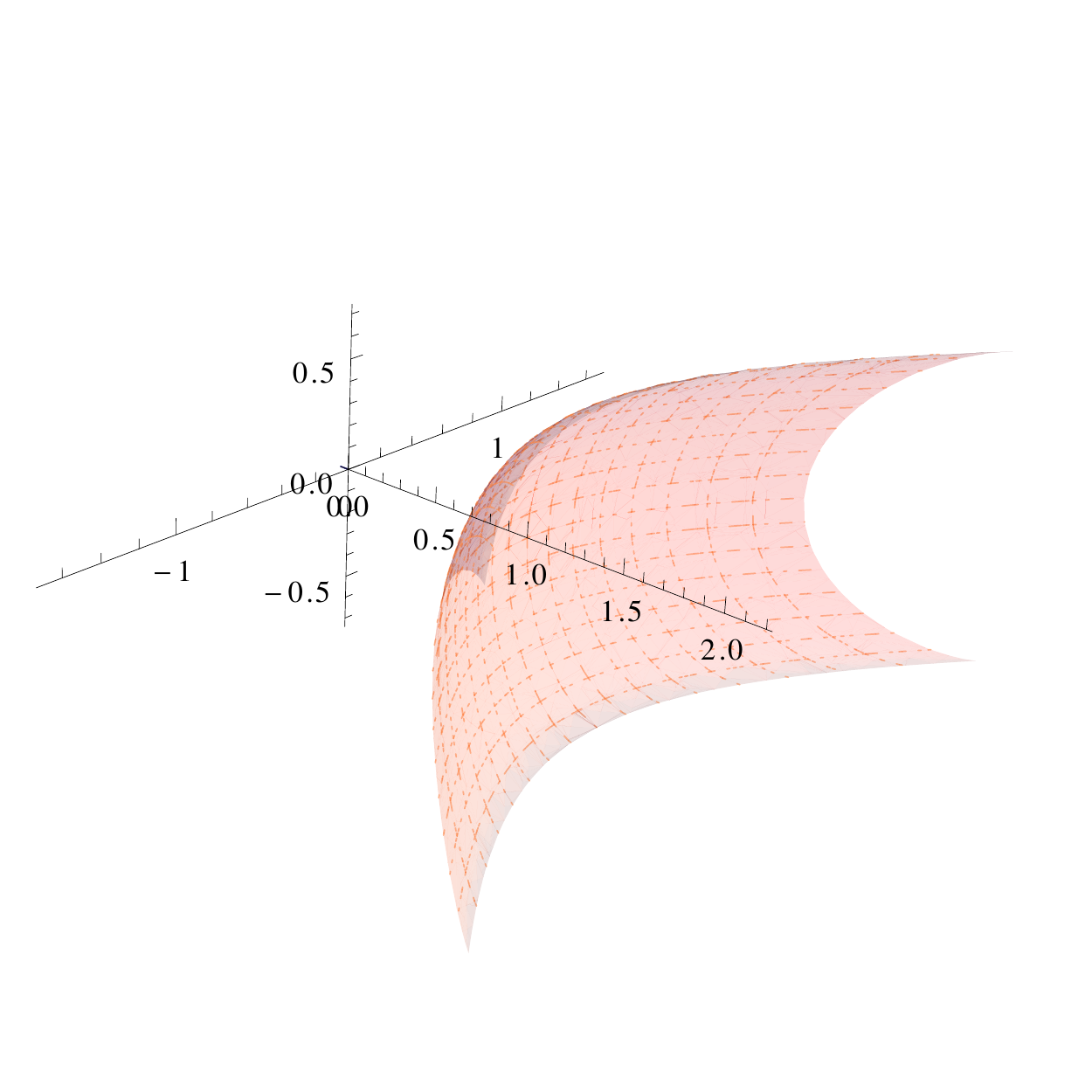} 
      \includegraphics[width=.3\textwidth]{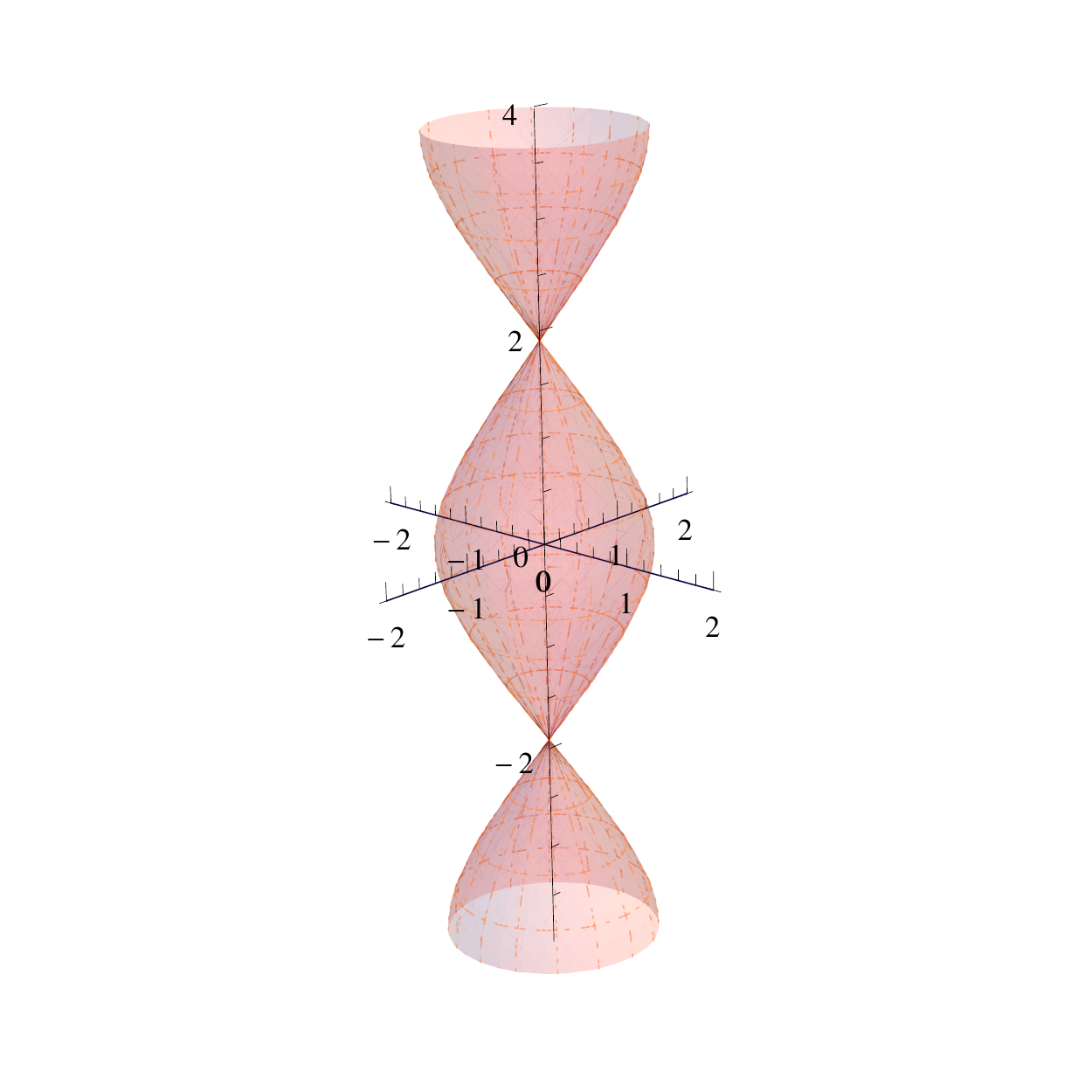}\label{fig:K-1L3}}\ 
    \caption{Constant Gaussian Curvature $-1$}
    \label{fig:K=-1}
  \end{figure}

  Finally, we note that the surfaces in figure
  \ref{fig:K=1}\subref{fig:K1L3} and figure
  \ref{fig:K=-1}\subref{fig:K-1R32} are isometric when seen as surfaces
  in $\R^3$ (with its induced metric) but in the pseudo-Riemannian
  ambients considered here they have rotational symmetry around axes
  with distinct causal types. The same holds for the corresponding
  surfaces in figures \ref{fig:K=1}\subref{fig:K1R32} and
  \ref{fig:K=-1}\subref{fig:K-1L3}.

\end{enumerate}

\bibliographystyle{amsplain}

\bigskip

\begin{minipage}{0.45\linewidth}
  \footnotesize{
    \noindent
    {\sc Ivo Terek Couto\\
    Instituto de Matem\'atica e Estat\'\i stica \\ Universidade de S\~ao
    Paulo}\\
  \emph{E-mail address}: {\tt terek@ime.usp.br}}
\end{minipage}
\begin{minipage}{0.45\linewidth}
  \footnotesize{
    \noindent
    {\sc Alexandre Lymberopoulos\\
      Instituto de Matem\'atica e Estat\'\i stica \\ Universidade de S\~ao
      Paulo}\\
    \emph{E-mail address}: {\tt lymber@ime.usp.br}}
\end{minipage}

\end{document}